\documentclass[12pt,draft]{amsart}
\usepackage{amsmath,amsthm,amssymb}

\newtheorem{thm}{Theorem}[section]

\newtheorem{lem}{Lemma}[section]
\newtheorem{prop}{Proposition}[section]
\newtheorem{conj}{Conjecture}[section]
\theoremstyle{definition}

\numberwithin{equation}{section}

\begin{document}

\begin{center}
 Symmetry-breaking bifurcation for the one-dimensional Liouville type equation
\end{center}

\vspace{2ex}

\begin{center}
 Satoshi Tanaka 
 \footnote{This work was supported by KAKENHI (23740113). \\ 
 \hfill \today}
 \\[2ex]
 Department of Applied Mathematics, Faculty of Science \\
 Okayama University of Science \\
 Ridaichou 1--1, Okayama 700--0005, Japan \\
 \texttt{tanaka@xmath.ous.ac.jp}
\end{center}

\vspace{3ex}

{\bf Abstract.}
The two-point boundary value problem for the
one-dimensional Liouville type equation 
\begin{equation*}
 \left\{
  \begin{array}{l}
   u'' + \lambda |x|^l e^u = 0, \quad x \in (-1,1), \\[1ex]
    u(-1) = u(1) = 0
  \end{array}
  \right.
\end{equation*}
is considered, where $\lambda>0$ and $l>0$.
In this paper, a symmetry-breaking result is obtained by using the Morse
index.
The problem 
\begin{equation*}
 \left\{
  \begin{array}{l}
   u'' + \lambda |x|^l (u+1)^p = 0, \quad x \in (-1,1), \\[1ex]
    u(-1) = u(1) = 0
  \end{array}
  \right.
\end{equation*}
is also considered, where $\lambda>0$, $l>0$, $p>1$ and $(p-1)l>4$.

\vspace{2ex}


\section{Introduction}

In this paper we consider the two-point boundary value problem for the
one-dimensional Liouville type equation 
\begin{equation}
 \left\{
  \begin{array}{l}
   u'' + \lambda |x|^l e^u = 0, \quad x \in (-1,1), \\[1ex]
    u(-1) = u(1) = 0,
    \label{LT1}
  \end{array}
  \right.
\end{equation}
where $\lambda>0$ and $l>0$.

Jacobsen and Schmitt \cite{JS} presented the exact multiplicity result of radial
solutions for the multi-dimensional problem 
\begin{equation}
 \left\{
  \begin{array}{cl}
   \Delta u + \lambda |x|^l e^u = 0 & \mbox{in}\ B, \\[1ex]
    u = 0 & \mbox{on}\ \partial B,
  \end{array}
  \right.
 \label{LT} 
\end{equation}
where $\lambda>0$, $l \ge 0$, $B:=\{ x \in {\bf R}^N : |x|<1 \}$
and $N \ge 1$.
In the case $N=1$, problem \eqref{LT} is reduced to \eqref{LT1}.
We note here that every solution of \eqref{LT} is positive in $B$,
by the strong maximum principle.
Jacobsen and Schmitt \cite{JS} proved the following (i)--(iii):
\begin{enumerate}
 \item if $1 \le N \le 2$, then there exists $\lambda_*>0$
       such that \eqref{LT} has
       exactly two radial solutions for $0<\lambda<\lambda_*$,
       a unique radial solution for $\lambda=\lambda_*$ and
       no radial solution for $\lambda>\lambda_*$;
 \item if $3 \le N < 10+4l$, then \eqref{LT} has
       infinitely many radial solutions when $\lambda=(l+2)(N-2)$ and
       a finite but large number of radial solutions when
       $|\lambda-(l+2)(N-2)|$ is sufficiently small;
 \item if $N \ge 10+4l$, then \eqref{LT} has
       a unique radial solution for $0<\lambda<(l+2)(N-2)$ and
       no radial solution for $\lambda\ge (l+2)(N-2)$.
\end{enumerate}
Recently, Korman \cite{Kor2} gave an alternative proof of (i)--(iii), and
his method is very interesting and easy to understand it.
Results (i)--(iii) were established by 
Liouville \cite{Lio},
Gel'fand \cite{Gel},	 
Joseph and Lundgren \cite{JL}
for problem \eqref{LT} with $l=0$, that is,
\begin{equation}
 \left\{
  \begin{array}{cl}
   \Delta u + \lambda e^u = 0 & \mbox{in}\ \Omega, \\[1ex]
    u = 0 & \mbox{on}\ \partial \Omega,
  \end{array}
  \right.
 \label{Liouville}
\end{equation}
when $\Omega=B$.

A celebrated theorem by Gidas, Ni and Nirenberg \cite{GNN} shows that every positive
solution of \eqref{Liouville} is radially symmetric when $\Omega=B$.
However, when $\Omega$ is an annulus 
$A:=\{ x \in {\bf R}^N : a < |x| < b \}$, $a>0$, problem
\eqref{Liouville} may has non-radial solutions.
Indeed, Nagasaki and Suzuki \cite{NS} found that large non-radial
solutions of \eqref{Liouville} when $N=2$ and $\Omega=A$.
More precisely, for each sufficiently large $\mu>0$, there
exist $\lambda>0$ and a non-radial solution $u$ of
\eqref{Liouville} such that $\int_A e^u dx=\mu$ when $N=2$ and $\Omega=A$.
Lin \cite{Lin} showed that \eqref{Liouville} has infinitely many
symmetry-breaking bifurcation points when $N=2$ and $\Omega=A$.
Dancer \cite{Dancer} proved that non-radial solution branches emanating
from the symmetry-breaking bifurcation points found by Lin \cite{Lin}
are unbounded.
Kan \cite{Kan1, Kan2} considered \eqref{Liouville} with $\Omega=A$ and
$N=2$ and investigated the structure of non-radial solutions
bifurcating from radial solutions in the case where $a$ is sufficiently
small.
More general potential and domain were considered by del Pino, Kowalczyk
and Musso \cite{dPKM}, and they constructed concentrating solutions.

Recently, Miyamoto \cite{Miya} proved the following result for \eqref{LT}.

\medskip

\noindent{\bf Theorem A (\cite{Miya}).} \it
Let $n_0$ be the largest integer that is smaller than $1+\frac{l}{2}$ and
let $\alpha_n:=2\log\frac{2l+4}{l+2-2n}$.
All the radial solutions of \eqref{LT} with $N=2$ can be written
explicitly as 
\[
 \lambda(\alpha) = 2 (l+2)^2 (e^{-\alpha/2}-e^{-\alpha}), \quad
 U(r;\alpha) = \alpha - 2\log ( 1+ (e^{\alpha/2}-1)r^{l+2} ).
\]
The radial solutions can be parameterized by the L$^\infty$-norm, it has
one turning point at $\lambda = \lambda(\alpha_0) = (l+2)/2$, and it
blows up as $\lambda \downarrow 0$. 
For each $n \in \{ 1, 2,\cdots, n_0\}$,
$(\lambda(\alpha_n),U(r;\alpha_n))$ is a symmetry breaking bifurcation
point from which an unbounded branch consisting of non-radial
solutions of \eqref{LT} with $N=2$ emanates, and $U(r;\alpha)$ is
nondegenerate if $\alpha\ne\alpha_n$, $n=0,1,\cdots, n_0$.
Each non-radial branch is in 
$(0,\lambda(\alpha_0))) \times \{u > 0\} \subset {\bf R} \times H_0^2(B)$.
\rm

\medskip

We return to problem \eqref{LT1}.
Korman \cite{Kor2} found the interesting property of radial solutions to
\eqref{LT}.
We will use it for the case $N=1$. 
Let $w$ be a unique solution of the initial value problem
\begin{equation*}
 \left\{
  \begin{array}{l}
   w'' + |x|^l e^w = 0, \quad x>0, \\[1ex]
   w(0)=w'(0)=0.
  \end{array}
 \right.
\end{equation*}
It is easy to show that
\begin{equation}
 w(x)<0, \ w'(x)<0, \ w''(x)<0 \quad \mbox{for}\ x>0 
  \label{w<0w'<0w''<0}
\end{equation}
and $\lim_{x\to\infty} w(x)=-\infty$.
Hence, there exists the inverse function $\eta$ of $-w(x)$.
It follows that $\eta \in C^2(0,\infty)$, $\eta(t)>0$,
$\eta'(t)>0$ for $t>0$, $\eta(0)=0$, and
$\lim_{t\to\infty} \eta(t)=\infty$.
We set
\begin{equation}
 \lambda(\alpha)=[\eta(\alpha)]^{l+2}e^{-\alpha}
  \label{Kormanlam}
\end{equation}
and
\begin{equation}
 U(x;\alpha)=w(\eta(\alpha)|x|)+\alpha.
 \label{Kormansol}
\end{equation}
By a direct calculation, we easily prove that, for each $\alpha>0$,  
$U(x;\alpha)$ satisfies $\|U\|_\infty=\alpha$ and is a positive
even solution of \eqref{LT1} at $\lambda=\lambda(\alpha)$.
Here and hereafter we use the notation: 
$\|u\|_\infty=\sup_{x \in [-1,1]}u(x)$.

\begin{lem}[Korman \cite{Kor2}]\label{Korman}
 For each $\alpha>0$, $U(x;\alpha)$ is a positive even solution of
 \eqref{LT1} at $\lambda=\lambda(\alpha)$ and $\|U\|_\infty=\alpha$.
\end{lem}

The author would like to call $U(x;\alpha)$ the {\it Korman solution} of
\eqref{LT1}.
Korman \cite{Kor2} also presented this kind of radial solutions to 
\begin{equation*}
 \left\{
  \begin{array}{cl}
   \Delta u + \lambda |x|^lf(u) = 0 & \mbox{in}\ B, \\[1ex]
    u = 0 & \mbox{on}\ \partial B
  \end{array}
  \right.
\end{equation*}
in the following cases: $f(u)=(u+1)^p$, $p>1$; $f(u)=(1-u)^{-p}$, $p>1$; 
$f(u)=e^{-u}$. 
By using Lemma \ref{Korman}, we can show the following result, which
will be shown in Section 2.

\begin{prop}\label{even}
 The functions $\lambda(\alpha)$ and $U(x;\alpha)$ satisfy 
 $\lambda(\alpha) \in$ \linebreak $C^2(0,\infty)$,
 $U(x;\alpha) \in C^2([-1,1] \times (0,\infty))$, and
 \begin{equation}
    \lim_{\alpha\to+0} \lambda(\alpha) 
  = \lim_{\alpha\to\infty} \lambda(\alpha) = 0.
   \label{limlambda}
 \end{equation} 
 Moreover, there exists $\alpha_*>0$ such that $\lambda'(\alpha)>0$ for
 $0<\alpha<\alpha_*$, $\lambda(\alpha_*)=0$ and $\lambda'(\alpha)<0$ for
 $\alpha>\alpha_*$.
\end{prop}

Hereafter, let $\alpha_*$ be as in Proposition \ref{even}.

Let $m(\alpha)$ be the Morse index of $U(x;\alpha)$, that
is, the number of negative eigenvalues $\mu$ to
\begin{equation}
 \left\{
 \begin{array}{l}
  \phi'' + \lambda(\alpha) |x|^l e^{U(x;\alpha)} \phi + \mu \phi = 0,
   \quad x \in (-1,1), \\[1ex]
  \phi(-1) = \phi(1) = 0.
 \end{array}
	\right.
  \label{Morse}
\end{equation}
A solution $U(x;\alpha)$ is said to be degenerate if
$\mu=0$ is an eigenvalue of \eqref{Morse}.
Otherwise, it is said to be nondegenerate.

We denote by $\mu_k(\alpha)$ the $k$-th eigenvalue of \eqref{Morse}.
We recall that 
\[
 \mu_1(\alpha) < \mu_2(\alpha) < \cdots < \mu_k(\alpha) <
 \mu_{k+1}(\alpha) < \cdots, \quad
 \lim_{k\to\infty} \mu_k(\alpha) = \infty,
\]
no other eigenvalues, an eigenfunction $\phi_k$ corresponding to
$\mu_k(\alpha)$ is unique up to a constant, and $\phi_k$ has exactly
$k-1$ zeros in $(-1,1)$.
We find that $\mu_k \in C(0,\infty)$. (See, for example, \cite{Kos}.)

The following theorem is the main result of this paper.

\begin{thm}\label{main}
 Let $(\lambda(\alpha),U(x;\alpha))$ be as in \eqref{Kormanlam}--\eqref{Kormansol} and
 let $\alpha_*>0$ be as in Proposition  \ref{even}.
 Then there exist constants $\alpha_1$, $\alpha_2$ and $\alpha_3$ such
 that $\alpha_*<\alpha_1 \le \alpha_2 \le \alpha_3$ and the following
 \textup{(i)--(vii)} hold\textup{:}
 \begin{enumerate}
  \item if $0<\alpha<\alpha_*$, then $m(\alpha)=0$ and $U(x;\alpha)$ is
	nondegenerate\textup{;}
  \item if $\alpha=\alpha_*$, then $m(\alpha)=0$ and $U(x;\alpha)$ is
	degenerate\textup{;}
  \item if $\alpha_*<\alpha<\alpha_1$, then $m(\alpha)=1$ and $U(x;\alpha)$
	is nondegenerate\textup{;}
  \item if $\alpha=\alpha_1$, then $m(\alpha)=1$ and $U(x;\alpha)$ is
	degenerate\textup{;}
  \item if $\alpha=\alpha_2$, then $m(\alpha)=1$, $U(x;\alpha)$ is
	degenerate and the point $(\lambda(\alpha_2),U(x;\alpha_2))$ is a non-even
	bifurcation point, that is, for each $\varepsilon>0$, there
	exists $(\lambda,u)$ such that $u$ is a positive non-even 
	solution of \eqref{LT1} and $|\lambda-\lambda(\alpha_2)|
	 + \|u-U(\,\cdot\,,\alpha_2)\|_\infty<\varepsilon$\textup{;}
  \item if $\alpha=\alpha_3$, then $m(\alpha)=1$ and $U(x;\alpha)$ is
	degenerate\textup{;}
  \item if $\alpha>\alpha_3$, then $m(\alpha)=2$ and $U(x;\alpha)$ is
	nondegenerate.
 \end{enumerate}
 Moreover, if $0<\lambda<\lambda(\alpha_3)$, then \eqref{LT1} has
 a positive non-even solutions $u$ which satisfies
 $\lim_{\lambda\to+0} \|u\|_\infty = \infty$.
\end{thm}

We note here that if $u$ is a non-even solution of \eqref{LT1}, then so
is $u(-x)$.

It is natural to expect that the following conjecture is true.

\begin{conj}
 In Theorem \eqref{main}, $\alpha_1=\alpha_2=\alpha_3$.
\end{conj}

Recalling the result by Jacobsen and Schmitt \cite{JS}, the structures of
radial solutions of \eqref{LT} with $N=2$ and even solutions of
\eqref{LT1} seem to be same.
However, in \cite{Miya} Miyamoto proved that the Morse index of the radial
solution increases by one when $\alpha$ passes each $\alpha_n$,
$n=0,1,2,\cdots,n_0$, where $\alpha_n$ is as in Theorem A. 
On the other hand, by Lemma \ref{mu3} below, the Morse index of even
solutions of \eqref{LT} is at most $2$ for each $l>0$.

When $N=2$, radial solutions of problems  \eqref{LT} and \eqref{Liouville}
can be written explicitly, and hence, Lin \cite{Lin} and Miyamoto
\cite{Miya} succeeded to find the bifurcation points.
That is difficult even if we know exact solutions, much more difficult
if we do not know them.
When $N \ne 2$, we do not know exact radial solutions of \eqref{LT} with
$l>0$.
However, recently Korman \cite{Kor2} found the solution \eqref{Kormansol}.
When $N=1$, the structure of eigenvalues
$\{\mu_k(\alpha)\}_{k=1}^\infty$ of \eqref{Morse} is well-known.
Combining these facts, we can show (i)--(iii) of Theorem \ref{main}.

Now we set
\[
 \psi(x;\alpha) := x U'(x;\alpha) + l+2
          = \eta(\alpha) |x| w'(\eta(\alpha)|x|) + l+2
\]
It is easy to check that the following result holds.

\begin{lem}\label{psi1}
 The function $\psi(x;\alpha)$
is a solution of the linearized equation
\begin{equation}\label{psi}
 \psi'' + \lambda(\alpha)|x|^l e^{U(x;\alpha)} \psi = 0.
\end{equation}
\end{lem}

Lemma \ref{psi1} was found by Korman \cite{Kor1} when $l=0$.
See also \cite[Proposition 2.2]{Kormanbook} and \cite[Lemma 5.1]{Kor2}.
From Lemma \ref{psi1}, it follows that $m(\alpha) \le 2$ for $\alpha>0$.
See Lemma \ref{mu3} below.
Moreover, by using the comparison function 
\[
 y(x)=xU(x;\alpha)-(x-1)^2U'(x;\alpha),
\]
which was introduced in \cite{T2013}, we can prove that
$m(\alpha) \ge 2$ for all sufficiently large $\alpha>0$.
See Lemma \ref{mu_2<0} below.
Then we can find a symmetry-breaking bifurcation point of \eqref{LT1},
by using the Leray-Schauder degree, and hence we will obtain (iv)--(vii)
of Theorem \ref{main}.

By using similar argument, we can establish a symmetry-breaking bifurcation
result for the problem
\begin{equation}
 \left\{
  \begin{array}{l}
   u'' + \lambda |x|^l (u+1)^p = 0, \quad x \in (-1,1), \\[1ex]
    u(-1) = u(1) = 0,
    \label{JL1}
  \end{array}
  \right.
\end{equation}
where $\lambda>0$, $l>0$ and $p>1$.

\begin{prop}\label{exactforJL1}
 There exists $\lambda_*>0$ such that \eqref{JL1} has exactly two
 positive even solutions for $0<\lambda<\lambda_*$, a unique positive
 even solution for $\lambda=\lambda_*$, and no positive even solution for
 $\lambda>\lambda_*$.
\end{prop}

\begin{prop}\label{even2}
 Proposition \ref{even} remains valid if \eqref{LT1} is replaced by
 \eqref{JL1}.
\end{prop}

\begin{thm}\label{main2}
 Assume that $(p-1)l>4$.
 Theorem \ref{main} remains valid if \eqref{LT1} is replaced by
 \eqref{JL1}.
\end{thm}

The proofs of Propositions \ref{exactforJL1}, \ref{even2} and
Theorem \ref{main2} will be given in Section 6.
In Section 2 we prove Proposition \ref{even} and study the eigenvalues
$\mu_1(\alpha)$ and $\mu_3(\alpha)$. 
In Section 3 we study $\mu_2(\alpha)$.
In Section 4 we give a criterion for the existence of one more positive
solution.
In Sections 5, we give a proof of Theorem \ref{main}.

\section{The first and third eigenvalues} 

In this section we study eigenvalues $\mu_1(\alpha)$ and $\mu_3(\alpha)$
of the linearized problem \eqref{Morse}.
We recall Lemmas \ref{Korman} and \ref{psi1}.
First we show Proposition \ref{even}.

\begin{proof}[Proof of Proposition \ref{even}]
 We recall \eqref{w<0w'<0w''<0}. 
 Since $w(0)=w'(0)=w''(0)$ \linebreak
 $=0$ and $\eta \in C^2(0,\infty)$, we conclude
 that  $\lambda(\alpha) \in C^2(0,\infty)$ and
 $U(x;\alpha) \in C^2([-1,1] \times (0,\infty))$.
 It is easy to see that $\lim_{\alpha\to+0} \lambda(\alpha)=0$.
 Since $w''(x)<0$ for $x>0$, we have
 \[
  w'(x) \le w'(1)<0, \quad x \ge 1.
 \]
 Integrating this inequality on $[1,x]$, we obtain
 \[
  w(x) \le w(1) + w'(1)(x-1) \le -c(x-1), \quad x\ge 1,
 \]
 where $c=-w'(1)>0$. 
 Letting $x=\eta(\alpha)$, we find that
 \[
   0<\lambda(\alpha)
   = [\eta(\alpha)]^{l+2} e^{-\alpha}
   = x^{l+2} e^{w(x)}
   \le x^{l+2} e^{-c(x-1)}, \quad x \ge 1,
 \]
 which means that $\lim_{\alpha\to\infty}\lambda(\alpha)=0$.
 We observe that
 \[
  \lambda'(\alpha)
  = [(l+2)\eta'(\alpha)-\eta(\alpha)][\eta(\alpha)]^{l+1}e^{-\alpha},
  \quad \alpha>0.
 \]
 Since $\eta'(\alpha)=-1/w'(\eta(\alpha))$, we have
 \[
  \lambda'(\alpha)
   = - [\eta(\alpha)w'(\eta(\alpha))+l+2]
    \frac{[\eta(\alpha)]^{l+1}}{e^\alpha w'(\eta(\alpha))}.
 \]
 Since
 \begin{equation}
  (xw'(x))' = w'(x) + x w''(x) = w'(x) - x^{l+1} e^{w(x)} < 0,
  \quad x > 0,
   \label{(xw')'<0} 
 \end{equation}
 there exists $\alpha_*>0$ such that 
 \begin{gather}
  \eta(\alpha)w'(\eta(\alpha))+l+2 >0, \quad 0 < \alpha < \alpha_*, 
  \label{etaw'+l+2>0} \\
  \eta(\alpha_*)w'(\eta(\alpha_*))+l+2=0, \label{etaw'+l+2=0} \\
  \eta(\alpha)w'(\eta(\alpha))+l+2 <0, \quad \alpha > \alpha_*.
  \label{etaw'+l+2<0}
 \end{gather}
 Consequently, we see that $\lambda'(\alpha)>0$ for
 $0<\alpha<\alpha_*$, $\lambda(\alpha_*)=0$ and $\lambda'(\alpha)<0$ for
 $\alpha>0$.
\end{proof}

Recalling \eqref{(xw')'<0} and the definition of $\psi(x;\alpha)$,
we conclude that $\psi(x;\alpha)$ is strictly decreasing in
$x\in(0,1]$ for each fixed $\alpha>0$.
Since $\psi(-x;\alpha)=\psi(x;\alpha)$, we find that
\[
 \min_{x \in [-1,1]} \psi(x;\alpha) = \psi(1;\alpha)
  = \eta(\alpha) w'(\eta(\alpha))+l+2.
\]
Then, by \eqref{etaw'+l+2>0}--\eqref{etaw'+l+2<0}, we have the following
result immediately.

\begin{lem}\label{psi2}
 The function $\psi(x;\alpha)$ satisfies the following \textup{(i)--(iii):}
 \begin{enumerate}
  \item if $0<\alpha<\alpha_*$, then $\psi(x;\alpha)>0$ for
	$x \in [-1,1]$\textup{;} 
  \item $\psi(x;\alpha_*)>0$ for $x \in (-1,1)$ and
	$\psi(-1;\alpha_*)=\psi(1;\alpha_*)=0$\textup{;} 
  \item if $\alpha>\alpha_*$, then $\psi(x;\alpha)$ has exactly two
	zeros in $(-1,1)$, \linebreak
	$\psi(-1;\alpha)<0$ and $\psi(1;\alpha)<0$. 
 \end{enumerate} 
\end{lem}

\begin{lem}\label{mu1}
 The first eigenvalue $\mu_1(\alpha)$ of \eqref{Morse} satisfies the
 following \textup{(i)--(iii):}
 \begin{enumerate}
  \item $\mu_1(\alpha)>0$ for $0<\alpha<\alpha_*$\textup{;} 
  \item $\mu_1(\alpha_*)=0$\textup{;} 
  \item $\mu_1(\alpha)<0$ for $\alpha>\alpha_*$\textup{.} 
 \end{enumerate} 
\end{lem}

\begin{proof}
 Let $\phi_1$ be an eigenfunction corresponding to $\mu_1(\alpha)$.
 We recall that $\phi_1(x) \ne 0$ on $(-1,1)$ and $\phi_1(-1)=\phi_1(1)=0$.

 (i)
 Assume that $\mu_1(\gamma_1)\le 0$ for some $\gamma_1 \in (0,\alpha_*)$.
 Sturm comparison theorem implies that every solution of \eqref{psi} at
 $\alpha=\gamma_1$ has at least one zero in $[-1,1]$.
 This contradicts (i) of Lemma \ref{psi2}.
 Hence, $\mu_1(\alpha)>0$ for $0<\alpha<\alpha_*$.

 (ii) From (ii) of Lemma \ref{psi2} it follows that $\psi(x;\alpha_*)$
 is an eigenfunction corresponding to $\mu_1(\alpha_*)$ and
 $\mu_1(\alpha_*)=0$.

 (iii)
 We assume that $\mu_1(\gamma_2) \ge 0$ for some $\gamma_2>\alpha_*$.
 Recalling (iii) of Lemma \ref{psi2} and using Sturm comparison theorem,
 we conclude that every solution of
 \[
  \phi''
   + [\lambda(\gamma_2) |x|^l e^{U(x;\gamma_2)} + \mu_1(\gamma_2)] \phi = 0
 \]
 has at least one zero in $(-1,1)$.
 On the other hand, the eigenfunction $\phi_1$ of \eqref{Morse}
 corresponding to $\mu_1(\beta)$ has no zero in $(-1,1)$, which is a
 contradiction.
 Consequently, $\mu_1(\alpha)<0$ for $\alpha>\alpha_*$.
 \end{proof}

\begin{lem}\label{mu3}
 The third eigenvalue $\mu_3(\alpha)$ of \eqref{Morse} is positive for
 $\alpha>0$.
\end{lem}

\begin{proof}
 Assume that $\mu_3(\alpha) \le 0$ for some $\alpha>0$.
 Let $\phi_3$ be an eigenfunction of \eqref{Morse} corresponding to
 $\mu_3(\alpha)$.
 Then $\phi_3(-1)=\phi_3(1)$ and $\phi_3$ has exactly two zeros in
 $(-1,1)$.
 Sturm comparison theorem shows that every solution of \eqref{psi}
 has at least three zeros in $[-1,1]$.
 Lemmas \ref{psi1} and \ref{psi2} imply that $\psi(x;\alpha)$ is a
 solution of \eqref{psi} and has at most two zeros in $[-1,1]$.
 This is a contradiction.
 Therefore, $\mu_3(\alpha)>0$ for $\alpha>0$.
\end{proof}

\section{The second eigenvalue}

The purpose of this section is to give a sufficient condition for the
second eigenvalue of the linearized problem to the following problem
\begin{equation}
 \left\{
 \begin{array}{l}
  u'' + \lambda |x|^l f(u) = 0, \quad x \in (-1,1), \\[1ex]
  u(-1)=u(1)= 0
 \end{array}
 \right.
 \label{GP1}
\end{equation}
to be negative, where $\lambda>0$, $l>0$,
$f \in C^1[0,\infty)$, $f(s)>0$ and $f'(s)\ge 0$ for $s>0$.
Namely we will show the following lemma.

\begin{lem}\label{mu_2<0}
 Assume that, for each sufficiently large $\alpha>0$, there exist
 $\lambda(\alpha)>0$ and $U(x;\alpha)$ such that $U(x;\alpha)$ is a 
 positive even solution of \eqref{GP1} at $\lambda=\lambda(\alpha)$.
 Assume moreover that 
 \begin{equation}
  \liminf_{s\to\infty} \frac{l(g(s)-1)-4}{g(s)+l+3}>0,
   \label{limg}
 \end{equation}
 where $g(s)=sf'(s)/f(s)$.
 Let $\mu_2(\alpha)$ be the second eigenvalue of
 \begin{equation}
  \left\{
   \begin{array}{l}
    \phi'' + \lambda(\alpha) |x|^l f'(U(x;\alpha)) \phi + \mu \phi = 0,
     \quad x \in (-1,1), \\[1ex]
     \phi(-1) = \phi(1) = 0.
   \end{array}
	\right.
  \label{Morse1}
 \end{equation}
 Then $\mu_2(\alpha)<0$ for all sufficiently large $\alpha>0$.
\end{lem}

To this end we need the following two lemmas.

\begin{lem}\label{phi_2}
 Let $\phi_2$ be an eigenfunction corresponding to the second eigenvalue
 $\mu_2(\alpha)$ of \eqref{Morse1}.
 Then $\phi_2$ is odd, $\phi_2(0)=\phi_2(1)=0$ and $\phi_2(x) \ne 0$ for
 $x \in (0,1)$.
\end{lem}

\begin{proof}
 Let $M_1$ be the first eigenvalue of 
 \begin{equation*}
  \left\{
   \begin{array}{l}
    \Phi'' + \lambda(\alpha) |x|^l f'(U(x;\alpha)) \Phi + M \Phi = 0,
     \quad x \in (0,1), \\[1ex]
     \Phi(0) = \Phi(1) = 0
   \end{array}
  \right.
 \end{equation*}
 and let $\Phi_1$ be an eigenfunction corresponding to $M_1$.
 Then $\Phi_1(0)=\Phi_1(1)=0$ and $\Phi_1(x) \ne 0$ on $(0,1)$.
 Set
 \[
  \Phi(x) = \left\{
   \begin{array}{ll}
    \Phi_1(x), & x \in [0,1], \\[1ex]
    -\Phi_1(-x), & x \in [-1,0).
   \end{array}
  \right.
 \]
 Noting that
 \[
  \lim_{x\to-0} \Phi''(x) = \lim_{x\to-0} (-\Phi_1''(-x)) = -\Phi_1''(0) = 0,
 \]
 we easily check that $\Phi$ is a solution of
 \[
   \left\{
   \begin{array}{l}
    \Phi'' + \lambda(\alpha) |x|^l f'(U(x;\alpha)) \Phi + M_1 \Phi = 0,
     \quad x \in (-1,1), \\[1ex]
     \Phi(-1) = \Phi(1) = 0,
   \end{array}
   \right.
 \]
 and $\Phi$ is odd, $\Phi(x)\ne 0$ on $(0,1)$ and $\Phi(0)=0$.
 Therefore, $M_1$ is an eigenvalue of \eqref{Morse1} and 
 $\Phi$ is an eigenfunction corresponding to $M_1$.
 Since $\Phi$ has exactly one zero in $(-1,1)$, $M_1$ must be $\mu_2$
 and hence $\phi_2(x)$ must be $c\Phi(x)$ for some $c \ne 0$.
\end{proof}

\begin{lem}\label{|w|}
 Assume that $w \in C[a,b]$ is positive and concave on $(a,b)$.
 Let $\rho \in (0,1/2)$.
 Then $w(x) \ge \rho \max_{\xi \in [a,b]} w(\xi)$ for
 $x \in [(1-\rho)a+\rho b, \rho a + (1-\rho)b]$.
\end{lem}

\begin{proof}
 We take $c \in [a,b]$ for which $w(c)=\max_{\xi \in [a,b]} w(\xi)$. 
 Then $w(c)>0$.
 Since $w$ is positive and concave on $(a,b)$, we have
 \[
  w(x) \ge \frac{w(c) (x-a)}{c-a} 
  \ge \frac{w(c) (x-a)}{b-a} =: l_1(x), \quad x \in [a,c],
 \]
 and
 \[
  w(x) \ge \frac{w(c) (b-x)}{b-c}
  \ge \frac{w(c) (b-x)}{b-a} =: l_2(x), \quad x \in [c,b].
 \]
 Hence $w(x) \ge \min \{ l_1(x), l_2(x) \}$ on $[a,b]$.
 We conclude that if $x \in [(1-\rho)a+\rho b, (a+b)/2]$, then
 \[
  \min \{ l_1(x), l_2(x) \} = l_1(x) 
  \ge l_1((1-\rho)a+\rho b)
  = \rho w(c),
 \]
 and if $x \in [(a+b)/2,\rho a+(1-\rho) b]$, then
 \[
  \min \{ l_1(x), l_2(x) \} = l_2(x) 
  \ge l_2(\rho a+(1-\rho)b)
  = \rho w(c).
 \]
 The proof is complete.
\end{proof}

Now we are ready to prove Lemma \ref{mu_2<0}.

\begin{proof}[Proof of Lemma \ref{mu_2<0}]
 Let $\alpha>0$ be sufficiently large.
 We use the following comparison function $y(x)$ introduced in \cite{T2013}:
 \[
  y(x)=xU(x;\alpha)-(x-1)^2U'(x;\alpha).
 \]
 This function $y(x)$ satisfies $y(0)=y(1)=0$, $y(x)>0$ on $(0,1)$, and
 \[
   y'' + \lambda(\alpha) |x|^l f'(U(x;\alpha)) y
   = \lambda(\alpha) x^{l-1} H(x;\alpha) f(U(x;\alpha)),
   \quad x \in (0,1],
 \]
 where
 \[
  H(x;\alpha) = [g(U(x;\alpha))+l+3] x^2 - 2(l+2)x + l.
 \]
 Let $\phi_2$ be an eigenfunction corresponding to $\mu_2(\alpha)$.
 From Lemma \ref{phi_2} it follows that
 $\phi_2(0)=\phi_2(1)=0$ and $\phi_2(x) \ne 0$ for $x \in (0,1)$.
 Without loss of generality, we may assume that $\phi_2(x)>0$ for
 $x \in (0,1)$ and $\max_{\xi \in [0,1]} \phi_2(\xi)=1$.
 We observe that
 \[
  (y'\phi_2 - y\phi_2')' 
   = \mu_2(\alpha) \phi_2 y
    + \lambda(\alpha) x^{l-1} H(x;\alpha) f(U(x;\alpha)) \phi_2,
   \quad x \in (0,1].
 \]
 Integrating this equality on $(0,1)$, we obtain
 \begin{multline}
  \mu_2(\alpha) \int_0^1 \phi_2(x) y(x) dx \\
  + \lambda(\alpha) \int_0^1 x^{l-1} H(x;\alpha) f(U(x;\alpha)) \phi_2(x) dx =0.
   \label{1}
 \end{multline}
 Since
 \begin{align*}
 H(x;\alpha)
 = & \ [g(U(x;\alpha))+l+3]
   \left( x - \frac{l+2}{g(U(x;\alpha))+l+3} \right)^2 \\
  & \ + \frac{l[g(U(x;\alpha))-1]-4}{g(U(x;\alpha))+l+3} \\
 \ge & \ \frac{l[g(U(x;\alpha))-1]-4}{g(U(x;\alpha))+l+3},
 \end{align*}
 we have
 \begin{multline}
  \int_0^1 x^{l-1} H(x;\alpha) f(U(x;\alpha)) \phi_2(x) dx \\
   \ge \int_0^1 x^{l-1}
   \frac{l[g(U(x;\alpha))-1]-4}{g(U(x;\alpha))+l+3}
  f(U(x;\alpha)) \phi_2(x) dx.
   \label{2}
 \end{multline}
 By \eqref{limg}, there exist $\delta>0$ and sufficiently large $s_0>0$
 such that
 \[
   \frac{l(g(s)-1)-4}{g(s)+l+3} \ge \delta, \quad s \ge s_0.
 \]
 Since $U''(x;\alpha)=-\lambda(\alpha)|x|^l f(U(x;\alpha))<0$ on
 $(0,1]$, we find that $U'(x;\alpha)$ is decreasing in $x \in (0,1]$.
 From $U'(0;\alpha)=0$, it follows that $U'(x;\alpha)<0$ for
 $x \in (0,1]$, which implies that $U(x;\alpha)$ is also decreasing
 in $x \in (0,1]$.
 Now let $\alpha>s_0$.
 Then there exists $x(\alpha) \in (0,1)$ such that
 $U(x;\alpha) \ge s_0$ for $x \in [0,x(\alpha)]$ and
 $U(x;\alpha) < s_0$ for $x \in (x(\alpha),1]$.
 Since $U(x;\alpha)$ is concave on $(0,1)$, we conclude that
 \[
  U(x;\alpha) \ge \alpha (1-x), \quad x \in [0,1],
 \]
 which shows that if $x \in [0,(\alpha-s_0)/\alpha]$, then
 $U(x;\alpha)\ge s_0$.
 Therefore, $x(\alpha) \ge (\alpha-s_0)/\alpha$, which implies
 \begin{equation}
  \lim_{\alpha\to\infty} x(\alpha) = 1.
   \label{x(a)->1}
 \end{equation}
 We take $s_1 \ge s_0$ for which $x(\alpha) \ge 3/4$ for $\alpha \ge s_1$.
 If $\alpha \ge s_1$, then
 \begin{multline}
   \int_0^{x(\alpha)} x^{l-1}
   \frac{l[g(U(x;\alpha))-1]-4}{g(U(x;\alpha))+l+3}
  f(U(x;\alpha)) \phi_2(x) dx \\
  \ge \int_0^{x(\alpha)} x^{l-1} \delta f(s_0) \phi_2(x) dx
  \ge \delta f(s_0) \int_{1/4}^{3/4} x^{l-1} \phi_2(x) dx.
   \label{3}
 \end{multline}
 Recalling $\max_{\xi \in [0,1]} \phi_2(\xi)=1$, we have
 \begin{align}
    \int_{x(\alpha)}^1 x^{l-1} &
   \frac{l[g(U(x;\alpha))-1]-4}{g(U(x;\alpha))+l+3}
  f(U(x;\alpha)) \phi_2(x) dx
  \label{4} \\
  \ge & \  -(l+4) \int_{x(\alpha)}^1
    x^{l-1} \frac{f(U(x;\alpha)) \phi_2(x)}{g(U(x;\alpha))+l+3} dx
  \nonumber \\
  \ge & -(l+4) \int_{x(\alpha)}^1 \frac{f(s_0)}{l+3} dx
  \nonumber \\
  = & \ -\frac{(l+4)f(s_0)}{l+3} (1-x(\alpha)), \quad \alpha \ge s_0.
  \nonumber 
 \end{align}
 
 Now we will show that there exists $s_2 \ge s_1$ such that
 $\mu_2(\alpha)<0$ for $\alpha \ge s_2$.
 Assume to the contrary that there exists $\{\alpha_n\}_{n=1}^\infty$
 such that $\mu_2(\alpha_n) \ge 0$ and $\alpha_n \ge s_1$ for
 $n \in {\bf N}$ and $\lim_{n\to\infty} \alpha_n=\infty$.
 Since $\phi_2(x)>0$ and
 $\phi_2''(x)=-|x|^l f'(U(x;\alpha_n)) \phi_2 - \mu_2(\alpha_n) \phi_2 \le 0$
 on $(0,1)$, we find that $\phi_2$ is concave on $(0,1)$ when
 $\alpha=\alpha_n$.
 From Lemma \ref{|w|} with $\rho=1/4$, $a=0$ and $b=1$, it follows that
 \begin{equation}
  \phi_2(x) \ge \frac{1}{4} \max_{\xi \in [0,1]} \phi_2(\xi)
   = \frac{1}{4} \quad \mbox{for} \
    x \in \left[\frac{1}{4},\frac{3}{4}\right], \ \alpha = \alpha_n.
   \label{5}
 \end{equation}
 Combining \eqref{1} with \eqref{2}, \eqref{3}--\eqref{5}, we conclude that
 \begin{align*}
 0 & \ge - \mu_2(\alpha_n) \int_0^1 \phi_2(x) y(x) dx \\
 & \ge \lambda(\alpha_n)f(s_0)
 \left[ \frac{\delta}{4} \int_{1/4}^{3/4} x^{l-1} dx
  -\frac{l+4}{l+3} (1-x(\alpha_n))
  \right],
 \end{align*}
 which implies
 \begin{equation*}
  \frac{l+4}{l+3} (1-x(\alpha_n))
  \ge \frac{\delta}{4} \int_{1/4}^{3/4} x^{l-1} dx >0,
  \quad n \in {\bf N}.
 \end{equation*} 
 This contradicts the fact \eqref{x(a)->1}.
 Consequently, there exists $s_2 \ge s_1$ such that
 $\mu_2(\alpha)<0$ for $\alpha \ge s_2$.
 This completes the proof of Lemma \ref{mu_2<0}.
\end{proof}

\section{Existence of another large solution}

In this section we give a criterion for the existence of a large
positive solution if there exists a positive even solution with the
Morse index 2. 

We consider the following problem
\begin{equation}
 \left\{
 \begin{array}{l}
  u'' + h(x)f(u) = 0, \quad x \in (-1,1), \\[1ex]
  u(-1)=u(1)= 0.
 \end{array}
 \right.
 \label{GP}
\end{equation}
Throughout this section, the following conditions are assumed to hold:
$h \in C[-1,1]$,
$h(x) \ge 0$ for $x \in [-1,1]$, $h(x)$ has at most finite zeros in
$[-1,1]$, $f \in C^1[0,\infty)$, $f(s)>0$, $f'(s)\ge0$ for $s \ge 0$,
and
\begin{equation}
  \lim_{s\to\infty} \frac{f(s)}{s} = \infty.
  \label{f(s)/s->oo}
\end{equation}
The purpose of this section is to prove the following existence result
which will be used in the proof of Theorem \ref{main}.

\begin{lem}\label{existPS}
 Assume that \eqref{GP} has a positive solution $U$ for which the Morse
 index of $U$ is $2$ and $U$ is nondegenerate.
 Then \eqref{GP} has a positive solution $u$ such that $u \not\equiv U$
 and $M f(\|u\|_\infty) > \| U \|_\infty$, where
 \[
  M = \int_{-1}^1 \int_{-1}^x h(t) dt dx.
 \]
\end{lem}

Here, the Morse index of $U$ is the number of negative eigenvalues $\mu$
of the problem
\begin{equation}
 \left\{
 \begin{array}{l}
  \phi'' + h(x) f'(U(x)) \phi + \mu \phi = 0,
   \quad x \in (-1,1), \\[1ex]
  \phi(-1) = \phi(1) = 0.
 \end{array}
	\right.
  \label{MorseofU}
\end{equation}

To prove Lemma \ref{existPS}, we extend the domain of $f(s)$ satisfying
$f \in C^1({\bf R})$ and $f(x)>0$ for $x \in {\bf R}$.
We also extend the domain of $h(x)$ satisfying
\begin{equation*}
 h \in C[-1,1], \ h(x) \ge 0 \ \mbox{for} \ x \ge -1 \ \mbox{and} \  
 \liminf_{x\to\infty} h(x)>0.
\end{equation*}

We denote by $u(x;\beta)$ the solution of the initial value problem
\begin{equation*}
 \left\{
  \begin{array}{l}
   u'' + h(x)f(u) = 0, \\[1ex]
   u(-1) = 0, \quad u'(-1)= \beta,
  \end{array}
 \right.
\end{equation*}
where $\beta>0$ is a parameter.
From a general theory on ordinary differential equations 
(see, for example, \cite{Har}), it follows that the solution
$u(x;\beta)$ exists on $[-1,\infty)$, it is unique, and 
$u(x;\beta)$, $u'(x;\beta)$ are $C^1$ functions on the set
$[-1,\infty)\times (0,\infty)$. 
By the same argument as in the proof of Lemma 2.1 in \cite{T2013}, we
easily see that, for each $\beta>0$, $u(x;\beta)$ has a
zero in $[-1,\infty)$. 
For each $\beta>0$, we denote the first zero of
$u(x;\beta)$ in $(-1,\infty)$ by $z(\beta)$.
Since $u(x;\beta)>0$ for $x \in (-1,z(\beta))$, by the
uniqueness of the initial value problem, we have
$u'(z(\beta);\beta)<0$.
Therefore we conclude that
\begin{equation*}
 u(z(\beta);\beta)=0, \quad u'(z(\beta);\beta)<0.
\end{equation*}
The implicit function theorem shows that $z \in C^1(0,\infty)$
and 
\begin{equation}
 z'(\beta)
 = - \frac{\displaystyle\frac{\partial u}{\partial \beta}(z(\beta);\beta)}
          {u'(z(\beta);\beta)}.
  \label{z'}
\end{equation}
By a general theory on ordinary differential equations 
(see, for example, \cite{Har}), we note that
$\frac{\partial u}{\partial \beta}(x;\beta)$ is
a unique solution of the initial value problem
\begin{equation}
 \left\{
  \begin{array}{l}
   v'' + h(x) f'(u)v = 0, \\[1ex]
   v(-1)=0, \quad v'(-1)=1,
  \end{array}
 \right.
 \label{wIe}
\end{equation}
where $u=u(x;\beta)$.

\begin{lem}\label{z(oo)<1}
 There exists $\beta^*>0$ such that $z(\beta) < 1$ for $\beta > \beta^*$.
\end{lem}

\begin{proof}
 Assume that there exists $\{\beta_n\}_{n=1}^\infty$ such that
 $\lim_{n\to\infty} \beta_n = \infty$, $\beta_n>0$ and
 $z(\beta_n) \ge 1$ for $n \in {\bf N}$.
 Set $u_n=u(x;\beta_n)$.
 Then $u_n(x)>0$ for $x \in (-1,1)$.
 Integrating $u_n''+ h(x)f(u_n)=0$ on $[-1,x]$ and integrating it on
 $[-1,1]$ again, we have
 \begin{equation}
  2\beta_n = \int_{-1}^1 \int_{-1}^x h(s) f(u_n(s)) ds dx
          \le f(\|u_n\|_\infty) \int_{-1}^1 \int_{-1}^x h(s)ds dx.
    \label{beta<|u|}
 \end{equation}
 Letting $n \to \infty$ in \eqref{beta<|u|}, we obtain
 \begin{equation}
  \lim_{n\to\infty} \|u_n\|_\infty=\infty.
   \label{un->oo}
 \end{equation}
 From Lemma \ref{|w|} with $a=-1$, $b=1$ and $\rho=1/4$, it follows that
 \begin{equation}
  u_n(x) \ge \frac{1}{4} \|u_n\|_\infty>0, \quad
   x \in \left[-\frac{1}{2},\frac{1}{2}\right].
   \label{un>1/4un}
 \end{equation}
 Let $\nu_1$ be the first eigenvalue of
 \[
  \left\{
  \begin{array}{l}
   \phi'' + \nu h(x) \phi = 0, \quad x \in (-1/2,1/2), \\[1ex]
   \phi(-1/2) = \phi(1/2) = 0.
  \end{array}
  \right.
 \]
 Then $\nu_1>0$. 
 By \eqref{f(s)/s->oo}, there exists $s_1>0$ such that
 \[
  \frac{f(s)}{s} > \nu_1, \quad s > s_1.
 \]
 By \eqref{un->oo}, there exists $n_1>0$ such that
 $\|u_n\|_\infty>4s_1$ for $n \ge n_1$.
 From \eqref{un>1/4un} it follows that if $n \ge n_1$, then
 \[
  h(x) \frac{f(u_n(x))}{u_n(x)} > \nu_1 h(x),
  \quad x \in \left[-\frac{1}{2},\frac{1}{2}\right].
 \]
 Since $u_n$ is a solution of
 \[
  u_n'' + h(x) \frac{f(u_n(x))}{u_n(x)} u_n = 0,
 \]
 Sturm comparison theorem implies that $u_n$ has at least one zero in
 $(-1/2,1/2)$.
 This contradicts \eqref{un>1/4un}.
 Therefore, there exists $\beta^*>0$ such that
 $z(\beta) < 1$ for $\beta > \beta^*$.
\end{proof}

\begin{lem}\label{z'(U'(-1))>0}
 Assume that \eqref{GP} has a positive solution $U$ for which the Morse
 index of $U$ is $2$ and $U$ is nondegenerate.
 Then $z'(U'(-1))>0$.
\end{lem}

\begin{proof}
 First we note that $U(x)=u(x;U'(-1))$ for $x \in [-1,1]$ and $z(U'(-1))=1$.
 Let $\mu_2$ and $\mu_3$ be the second and third eigenvalues of
 \eqref{MorseofU}, respectively.
 Then $\mu_2<0<\mu_3$.
 Let $\phi_2$ and $\phi_3$ be eigenfunctions
 corresponding to $\mu_2$ and $\mu_3$, respectively.
 Let $v$ be the solution of \eqref{wIe} with $u=U$.
 We recall that $v(x)\equiv\frac{\partial u}{\partial \beta}(x;U'(-1))$.
 Since $\phi_2$ has exactly one zero in $(-1,1)$, Sturm comparison
 theorem implies that $v$ has at least two zeros in $(-1,1)$.
 If $v$ has three zeros in $(-1,1]$, then, by Sturm comparison theorem
 again, $\phi_3$ has at least three zeros in $(-1,1)$, which is a
 contradiction.
 Therefore, $v$ has exactly two zeros in $(-1,1)$ and $v(1) \ne 0$.
 Since $v'(-1)=1>0$, we conclude that $v(1)>0$.
 Since $\frac{\partial u}{\partial \beta}(z(U'(-1));U'(-1))=v(1)>0$
 and $u'(z(U'(-1));U'(-1))=U'(1)<0$, by \eqref{z'}, we obtain
 $z'(U'(-1))>0$.
 \end{proof}

Now we are ready to show Lemma \ref{existPS}.

\begin{proof}[Proof of Lemma \ref{existPS}]
 By Lemma \ref{z(oo)<1}, there exists $\beta^*>0$ such that
 $z(\beta)<1$ for $\beta>\beta^*$.
 Hence, by Lemma \ref{z'(U'(-1))>0} and $z(U'(-1))=1$, there exists
 $\beta_0 \in (U'(-1),\beta^*)$ such that $z(\beta_0)=1$.
 Then $u:=u(x;\beta_0)$ is a positive solution of \eqref{GP}.
 Since $u'(-1)=\beta_0>U'(-1)$, we conclude that $u \not\equiv U$, by
 the uniqueness of the initial value problem.
 Integrating $u''+h(x)f(u)=0$ on $[-1,x]$ and
 integrating it on $[-1,1]$ again, we have
 \begin{equation*}
  2 \beta_0
   = \int_{-1}^1 \int_{-1}^x h(t) f(u(t)) dt dx \\
   \le M f(\|u\|_\infty).
 \end{equation*}
 Let $c \in (-1,1)$ satisfy $U(c)=\|U\|_\infty$.
 Since $U$ is concave on $(-1,1)$, we have
 \[
  U(x) \ge \frac{\| U \|_\infty}{c+1} (x+1), \quad x \in [-1,c].
 \]
 Hence,
 \[
  U'(-1)
   = \lim_{x\to-1} \frac{U(x)-U(-1)}{x+1}
   = \lim_{x\to-1} \frac{U(x)}{x+1}
   \ge \frac{\| U \|_\infty}{c+1}
   \ge \frac{\| U \|_\infty}{2}.
 \]
 Consequently,
 \[
  M f(\|u\|_\infty) \ge 2\beta_0 > 2U'(-1) \ge \| U \|_\infty.
 \]
\end{proof}

\section{Proof of the main result}

In this section we give a proof of Theorem \ref{main}.

Lemma \ref{mu1} means (i) and (ii) of Theorem \ref{main}.
Moreover, since $\mu_2(\alpha)>\mu_1(\alpha)$, we have
\begin{equation}\label{mu2>0}
 \mu_2(\alpha)>0, \quad 0 < \alpha \le \alpha_*. 
\end{equation}
When $f(s)=e^s$, we have $g(s):=sf'(s)/f(s)=s$ and
\[
 \liminf_{s\to\infty} \frac{l(g(s)-1)-4}{g(s)+l+3}=l>0.
\]
From Lemma \ref{mu_2<0} it follows that $\mu_2(\alpha)<0$ for all
sufficiently large $\alpha>0$.
Hence, by \eqref{mu2>0}, there exist $\alpha_1$ and $\alpha_3$ such that
$\alpha_*<\alpha_1 \le \alpha_3$ such that
\begin{equation}\label{mu_2>=0}
 \mu_2(\alpha_1)=0, \quad \mu_2(\alpha)>0, \ \ 0<\alpha<\alpha_1
\end{equation}
and
\begin{equation*}
 \mu_2(\alpha_3)=0, \quad \mu_2(\alpha)<0, \ \ \alpha>\alpha_3.
\end{equation*}
Therefore, Lemma \ref{mu3} implies (vi) and (vii) of Theorem \ref{main}.
From Lemma \ref{mu1} and \eqref{mu_2>=0}, it follows that (iii) and (iv) of
Theorem \ref{main} hold.

Now we will show (v).
To this end, we define $T(\alpha,v)$ by
\[
 T(\alpha,v)
 = \int_{-1}^1 G(x,y) \lambda(\alpha) |y|^l e^{U(y;\alpha)} (e^{v(y)}-1) dy,
\]
where $G(x,y)$ is a Green's function of
the operator $L[v]=-v''$ with $v(-1)=v(1)=0$: 
\[
 G(x,y) =
  \left\{
   \begin{array}{ll}
    (1+x)(1-y)/2, & -1 \le x \le y \le 1, \\[1ex]
    (1-x)(1+y)/2, & -1 \le y \le x \le 1.
   \end{array}
  \right.
\]
Then \eqref{LT1} can be rewritten as
\begin{equation}
 v-T(\alpha,v)=0.
  \label{v-T=0}
\end{equation}
We note that \eqref{v-T=0} has a solution $v=0$ and
if $v$ is a solution of \eqref{v-T=0}, then $u(x)=U(x;\alpha)+v(x)$
is a solution of \eqref{LT1} at $\lambda=\lambda(\alpha)$.
 
\begin{lem}\label{gamma=m}
 Let $\gamma(\alpha)$ be the sum of algebraic multiplicities of all the
 eigenvalues of $T_v'(\alpha,0)$ contained in $(1,\infty)$.
 Then $m(\alpha)=\gamma(\alpha)$.
\end{lem}

\begin{proof}
 First we note that an eigenvalue $\nu$ of $T_v'(\alpha,0)$ with $\nu>1$
 is an eigenvalue 
 of the problem
 \begin{equation}
   \left\{
  \begin{array}{l}
   \psi'' + \frac{1}{\nu} \lambda(\alpha) |x|^l e^{U(x;\alpha)} \psi = 0,
    \quad x \in (-1,1), \\[1ex]
    \psi(-1) = \psi(1) = 0.
  \end{array}
   \right.
   \label{T'}
 \end{equation}
 We conclude that \eqref{T'} has eigenvalues
 $\{\nu_k(\alpha)\}_{k=1}^\infty$ for which
 \[
  \nu_1(\alpha) > \nu_2(\alpha) > \cdots > \nu_k(\alpha) >
  \nu_{k+1}(\alpha) > \cdots > 0, \quad
  \lim_{k\to\infty} \nu_k(\alpha) = 0,
 \]
 no other eigenvalues, an eigenfunction $\psi_k$ corresponding to
 $\nu_k(\alpha)$ is unique up to a constant, and $\psi_k$ has exactly
 $k-1$ zeros in $(-1,1)$.

 Next we will show that $\nu_k(\alpha)>1$ implies $\mu_k(\alpha)<0$.
 Assume that $\nu_k(\alpha)>1$ and $\mu_k(\alpha)\ge 0$.
 Then
 \[
 \frac{1}{\nu_k(\alpha)} \lambda(\alpha) |x|^l e^{U(x;\alpha)} <
  \lambda(\alpha) |x|^l e^{U(x;\alpha)} + \mu_k(\alpha).
 \]
 Sturm comparison theorem implies that an eigenfunction $\phi_k$
 corresponding to $\mu_k(\alpha)$ has at least $k$ zeros in
 $(-1,1)$.
 This is a contradiction.
 Hence, $\nu_k(\alpha)>1$ implies $\mu_k(\alpha)<0$.
 
 Finally we will prove that $\mu_k(\alpha)<0$ implies $\nu_k(\alpha)>1$.
 Assume that $\mu_k(\alpha)<0$ and $\nu_k(\alpha)\le 1$.
 Since
 \[
 \frac{1}{\nu_k(\alpha)} \lambda(\alpha) |x|^l e^{U(x;\alpha)} >
  \lambda(\alpha) |x|^l e^{U(x;\alpha)} + \mu_k(\alpha),
 \]
 By Sturm comparison theorem again, we conclude that an eigenfunction
 $\psi_k$ corresponding to $\nu_k(\alpha)$ has at least one $k$ zeros in
 $(-1,1)$, which is a contradiction.
 Then $\mu_k(\alpha)<0$ implies $\nu_k(\alpha)>1$.

 Consequently, $m(\alpha)=\gamma(\alpha)$.
\end{proof}

\begin{lem}\label{deg}
  For each sufficiently small $\varepsilon>0$, there exists 
  $(\alpha_\varepsilon,v_\varepsilon)$ such that
  $\alpha_1-\varepsilon \le \alpha_\varepsilon \le \alpha_3+\varepsilon$,
  $v_\varepsilon \in C[-1,1]$, and 
  \[
   v_\varepsilon-T(\alpha_\varepsilon,v_\varepsilon)=0, \quad
   \| v_\varepsilon \|_\infty \le \varepsilon, \quad
   v_\varepsilon \ne 0.
  \]  
\end{lem}

\begin{proof}
 Assume there exists $\varepsilon>0$ such that
 \[
  v-T(\alpha,v) \ne 0 \quad \textup{for} \ 
  \alpha_1-\varepsilon \le \alpha \le \alpha_3+\varepsilon, \
  v \in B_\varepsilon(0)-\{0\},
 \]
 where $B_\varepsilon(0)=\{ v \in C[-1,1] : \| v \|_\infty < \varepsilon \}$.
 Since $T(\alpha,v)$ is a compact operator on $C[0,1]$ for each fixed
 $\alpha>0$, Leray-Schauder degree
 $\mbox{deg}_{\rm LS}(I-T(\alpha,\,\cdot\,),B_\varepsilon(0),0)$
 is well defined in $C[-1,1]$.
 By the homotopy invariance of the Leray Schauder degree, we conclude that
 \[
  \mbox{deg}_{\rm LS}(I-T(\alpha,\,\cdot\,),B_\varepsilon(0),0) \
  \mbox{is\ constant\ for} \
  \alpha_1-\varepsilon \le \alpha \le \alpha_3+\varepsilon.
 \]
 It is known (for example, 
 \cite[Theorem 3.20]{AM}) that
 \[
  \mbox{deg}_{\rm LS}(I-T(\alpha_1-\varepsilon,\,\cdot\,),B_\varepsilon(0),0)
  = (-1)^{\gamma(\alpha_1-\varepsilon)}
 \]
 and
 \[
  \mbox{deg}_{\rm LS}(I-T(\alpha_3+\varepsilon,\,\cdot\,),B_\varepsilon(0),0)
  = (-1)^{\gamma(\alpha_3+\varepsilon)},
 \]
 where $\gamma(\alpha)$ is as in Lemma \ref{gamma=m}.
 Lemma \ref{gamma=m} implies that
 \[
  \gamma(\alpha_1-\varepsilon) = m(\alpha_1-\varepsilon) = 1
 \]
 and
 \[
  \gamma(\alpha_3+\varepsilon) = m(\alpha_3+\varepsilon) = 2,
 \]
 which means that
 \[
  \mbox{deg}_{\rm LS}(I-T(\alpha_1-\varepsilon,\,\cdot\,),B_\varepsilon(0),0) =-1
 \]
 and
 \[
  \mbox{deg}_{\rm LS}(I-T(\alpha_3+\varepsilon,\,\cdot\,),B_\varepsilon(0),0) =1.
 \]
 This contradicts the homotopy invariance of the Leray-Schauder degree.
\end{proof}

 Now we are ready to prove (v) of Theorem \ref{main}.
 Let $\{(\alpha_\varepsilon,v_\varepsilon)\}$ be as in Lemma \ref{deg}.
 Since $\alpha_\varepsilon \in [\alpha_1-\varepsilon,\alpha_3+\varepsilon]$, 
 there exists a subsequence of $\{(\alpha_\varepsilon,v_\varepsilon)\}$,
 again denoted by $\{(\alpha_\varepsilon,v_\varepsilon)\}$ such that
 \[
  \alpha_\varepsilon \to \alpha_2, \quad
  v_\varepsilon \to 0 \quad \mbox{as} \ \varepsilon \to +0
 \]
 for some $\alpha_2 \in [\alpha_1,\alpha_3]$.
 Consequently, $(\lambda(\alpha_2),U(x;\alpha_2))$ is a bifurcation
 point.
 Clearly, $U(x;\alpha_2)$ is degenerate.
 Moreover, $u_\varepsilon(x):=U(x;\alpha_\varepsilon)+v_\varepsilon(x)$
 is a solution of \eqref{LT1}.
 By recalling that $\lambda'(\alpha)<0$ for $\alpha>\alpha_*$,
 there is no even solution $u$ of \eqref{LT1} at
 $\lambda=\lambda(\alpha)$ such that
 $\|u\|_\infty>\alpha_*$ except $U(x;\alpha)$.
 Since
 \[
  \alpha_\varepsilon-\varepsilon
  \le \|u_\varepsilon\|_\infty \le \alpha_\varepsilon+\varepsilon,
 \]
 we conclude that $u_\varepsilon$ is a non-even solution of \eqref{LT1},
 and hence (v) of Theorem \ref{main} holds.

 Finally, we give a proof of the remaining part of Theorem \ref{main},
 that is, we will show that, for each
 $\lambda \in (0,\lambda(\alpha_3))$, problem \eqref{LT1} has a positive
 non-even solution $u(x)$ which satisfies
 $\lim_{\lambda\to+0} \|u\|_\infty = \infty$.
 Let $\lambda \in (0,\lambda(\alpha_3))$.
 Then, by Proposition \ref{even}, there exists $\alpha_\lambda>\alpha_3$ such
 that $\lambda(\alpha_\lambda)=\lambda$,
 $\lim_{\lambda\to+0} \alpha_\lambda=\infty$
 and
 $\lim_{\lambda\to+0} \lambda(\alpha_\lambda)=0$.
 From (vii) of Theorem \ref{main} it follows that $m(\alpha_\lambda)=2$ and
 $U(x;\alpha_\lambda)$ is nondegenerate.
 Lemma \ref{existPS} implies that \eqref{LT1} has a positive solution $u$
 such that $u(x) \not \equiv U(x;\alpha_\lambda)$ and 
 \[
  \lambda(\alpha_\lambda) M e^{\|u\|_\infty} > \alpha_\lambda
 \]
 for some constant $M>0$, which shows
 $\lim_{\lambda\to+0} \|u\|_\infty=\infty$.
 Recalling that \eqref{LT1} has at most two positive even solutions, we
 conclude that $u$ is a positive non-even solution.
 This completes the proof of Theorem \ref{main}.

\section{Proof of the second main result}

In this section we prove Propositions \ref{exactforJL1}, \ref{even2} and
Theorem \ref{main2}.

Let $w$ be a unique solution of the initial value problem
\begin{equation*}
 \left\{
  \begin{array}{l}
   w'' + |x|^l (w+1)^p = 0, \quad x>0, \\[1ex]
    w(0)=w'(0)=0.
  \end{array}
 \right.
\end{equation*}
Since $w$ is concave when $w(x)>-1$, there exist
$x_1>0$ such that $-1<w(x)<0$, $w'(x)<0$, $w''(x)<0$ for $x \in (0,x_1)$, 
$w(x_1)=-1$, and $w'(x_1)<0$.
Hence, there exists the inverse function $\eta$ of $-w(x)$.
It follows that $\eta \in C^2(0,1]$, $\eta(t)>0$,
$\eta'(t)>0$ for $t\in (0,1]$, $\eta(0)=0$, and $\eta(1)=x_1$.
We set
\begin{equation}
 \lambda(\alpha)
  =(\alpha+1)^{1-p}
   \left[ \eta\left(\frac{\alpha}{\alpha+1}\right) \right]^{l+2}
  \label{Kormanlam2}
\end{equation}
and
\begin{equation}
 U(x;\alpha)
  =(\alpha+1)w\left( \eta\left(\frac{\alpha}{\alpha+1}\right)|x|\right)+\alpha.
 \label{Kormansol2}
\end{equation}
Then $(\lambda(\alpha),U(x;\alpha))$ is a Korman solution of \eqref{JL1},
that is, for each $\alpha>0$,
$U(x;\alpha)$ satisfies $\|U\|_\infty=\alpha$ and is a positive
even solution of \eqref{JL1} at $\lambda=\lambda(\alpha)$.
The form of $U(x;\alpha)$ is not exactly same as in the paper by Korman
\cite{Kor2}, but they are essentially same.

\begin{proof}[Proof of Proposition \ref{even2}]
 By the definition, it is easy to check that
 $\lambda(\alpha) \in C^2(0,\infty)$,
 $U(x;\alpha) \in C^2([-1,1] \times (0,\infty))$ and \eqref{limlambda}
 holds, because of $p>1$.
 Set
 \[
  \beta = \eta\left( \frac{\alpha}{\alpha+1} \right).
 \]
 Then $-w(\beta)=\alpha/(\alpha+1)$, that is, $\alpha = -w(\beta)/(w(\beta)+1)$.
 Hence we have
 \[
  \lambda(\alpha) = (w(\beta)+1)^{p-1} \beta^{l+2}.
 \]
 We note that
 \[
  \frac{d \beta}{d \alpha}
  = \frac{1}{(\alpha+1)^2} \eta'\left( \frac{\alpha}{\alpha+1} \right) >0,
  \quad \alpha>0.
 \]
 We observe that
 \[
  \lambda'(\alpha)
   = (w(\beta)+1)^{p-2} \beta^{l+1} [(p-1)\beta
   w'(\beta)+(l+2)(w(\beta)+1)] \frac{d \beta}{d \alpha}.
 \]
 We also note that 
 \begin{equation}
  W(x):= (p-1)xw'(x)+(l+2)(w(x)+1)
   \label{W}
 \end{equation}
 is strictly decreasing on $(0,x_1)$, since
 \[
  (xw'(x))' = w'(x) + xw''(x) <0, \quad x \in (0,x_1).
 \]
 Since $W(0)=l+2>0$ and $W(x_1)=(p-1)x_1w'(x_1)<0$, there exists
 $\beta_* \in (0,x_1)$ such that
  \begin{gather}
   W(x)>0, \quad 0 < x < \beta_*, 
    \label{W(x)>0} \\
   W(\beta_*)=0,
    \label{W(x)=0} \\
   W(x)<0, \quad \beta_*<x<x_1. 
  \label{W(x)<0}
  \end{gather}
 Set $\alpha_*=-w(\beta_*)/(w(\beta_*)+1)$.
 Then we conclude that $\lambda'(\alpha)>0$ for
 $0<\alpha<\alpha_*$, $\lambda(\alpha_*)=0$ and
 $\lambda'(\alpha)<0$ for $\alpha>\alpha_*$.
\end{proof}

To prove Proposition \ref{exactforJL1}, we need the following lemma.

\begin{lem}\label{uniqueJL1}
 For each $\alpha>0$, there exists a unique $(\lambda,u)$ such that
 $\lambda>0$ and $u$ is a positive even solution of \eqref{JL1} and
 $\|u\|_\infty=\alpha$.
 In particular, all positive even solutions of \eqref{JL1} can be
 written as \eqref{Kormanlam2}--\eqref{Kormansol2}.
\end{lem}

\begin{proof}
 Let $\alpha>0$ be fixed.
 We consider the initial value problem
 \begin{equation}\label{u(x;l)}
  \left\{
   \begin{array}{l}
    u'' + \lambda |x|^l (u+1)^p = 0, \\[1ex]
     u(0)=\alpha, \quad u'(0)=0.
   \end{array}
 \right.
 \end{equation}
 We note that
 \[
 u(x;\lambda)
  :=(\alpha+1)w( \lambda^{\frac{1}{l+2}} (\alpha+1)^{\frac{p-1}{l+2}} |x| )
    + \alpha
 \]
 is a solution of \eqref{u(x;l)}.
 By the uniqueness of the initial value problem, we conclude that
 $u(x;\lambda)$ is a unique solution of \eqref{u(x;l)}.
 We note that $u(1;\lambda)=0$ if and only if $\lambda=\lambda(\alpha)$.
 It follows that $u(x;\lambda)$ is a positive even solution of
 \eqref{JL1} if and only if $\lambda=\lambda(\alpha)$, which means that
 there exists a unique $\lambda>0$ such that $u(x;\lambda)$ is a
 solution of \eqref{JL1}.
 When $\lambda=\lambda(\alpha)$, we find that $u(x;\lambda)=U(x;\alpha)$.
\end{proof}

\begin{proof}[Proof of Proposition \ref{exactforJL1}]
 Set $\lambda_*=\lambda(\alpha_*)$.
 Then Proposition \ref{exactforJL1} follows immediately from Proposition
 \ref{even2} and Lemma \ref{uniqueJL1}.
\end{proof}

Now we set
\begin{align*}
 \psi(x;\alpha)
   & := x U'(x;\alpha) + \frac{l+2}{p-1}[U(x;\alpha)+1] \\
   & \phantom{:} = \frac{\alpha+1}{p-1}
         W\left(\eta\left(\frac{\alpha}{\alpha+1}\right) |x| \right),
\end{align*}
where $W$ is the function defined by \eqref{W}.
Then it is easy to check that $\psi(x;\alpha)$ is a solution of the
linearized equation
\begin{equation*}
 \psi'' + \lambda(\alpha) |x|^l p (U(x;\alpha)+1)^{p-1} \psi = 0.
\end{equation*}
Recalling that $W(x)$ is strictly decreasing in $x \in (0,x_1)$, we
conclude that $\psi(x;\alpha)$ is also strictly decreasing in
$x \in (0,1)$ for each fixed $\alpha>0$, and hence
\[
 \min_{x \in [-1,1]} \psi(x;\alpha) = \psi(1;\alpha)
  = \frac{\alpha+1}{p-1}
         W\left(\eta\left(\frac{\alpha}{\alpha+1}\right) \right).
\]
Hereafter, let $\mu_k(\alpha)$ be the $k$-th eigenvalue of
\begin{equation}
 \left\{
 \begin{array}{l}
  \phi'' + \lambda(\alpha) |x|^l p(U(x;\alpha)+1)^{p-1} \phi + \mu \phi = 0,
   \quad x \in (-1,1), \\[1ex]
  \phi(-1) = \phi(1) = 0.
 \end{array}
	\right.
  \label{Morse2}
\end{equation}
By \eqref{W(x)>0}--\eqref{W(x)<0}, in the same way as in Section 2,
we have the following result.

\begin{lem}
 The following \textup{(i)--(iv) hold:}
 \begin{enumerate}
  \item $\mu_1(\alpha)>0$ for $0<\alpha<\alpha_*$\textup{;} 
  \item $\mu_1(\alpha_*)=0$\textup{;} 
  \item $\mu_1(\alpha)<0$ for $\alpha>\alpha_*$\textup{;}
  \item $\mu_3(\alpha)>0$ for $\alpha>0$\textup{.}
 \end{enumerate} 
\end{lem}

When $f(s)=(s+1)^p$, we have $g(s)=sf'(s)/f(s)=ps/(s+1)$ and then
\begin{equation*}
 \lim_{s\to\infty} \frac{l(g(s)-1)-4}{g(s)+l+3}= \frac{l(p-1)-4}{p+l+3}.
\end{equation*}
Therefore, if $(p-1)l>4$, then Lemma \ref{mu_2<0} shows that
$\mu_2(\alpha)<0$ for all sufficiently large $\alpha>0$.

In the same way as in Section 5, we can show (i)--(vii) of Theorem
\ref{main2}.
By using Lemma \ref{existPS} and the same argument as in Section 5,
we conclude that if $0<\lambda<\lambda(\alpha_3)$, then
\eqref{JL1} has a positive non-even solution $u$ such that
$\lim_{\lambda\to+0} \|u\|=\infty$.
This completes the proof of Theorem \ref{main2}.

\end{document}